\documentclass[10pt, a4paper, journal, twocolumn]{IEEEtran}

\IEEEoverridecommandlockouts                              

\usepackage{graphicx}
\usepackage{amsmath}
\usepackage{amssymb}

\usepackage{dsfont}
\usepackage{caption}
\usepackage{subcaption}

\usepackage{amsthm}
\newtheorem{theorem}{Theorem}

\newtheorem{lemma}[theorem]{Lemma}

\newtheorem{assumption}{Assumption}

\newtheorem{example}{Example}
\newtheorem{remark}{Remark}

\usepackage{xcolor}

\newcommand{\bal}[1] {\ensuremath{\left(\begin{array}{#1}}}
\newcommand{\ear} {\ensuremath{\end{array}\right)}}

\newcommand{\bals}[1] {\ensuremath{\left[\begin{array}{#1}}} 
\newcommand{\ears} {\ensuremath{\end{array} \right] }} 

\DeclareMathOperator{\trace}{trace}
\DeclareMathOperator{\sign}{sign}

\DeclareMathOperator{\diag}{diag}
\DeclareMathOperator{\blcdiag}{blockdiag}

\DeclareMathOperator{\sinc}{sinc}

\newcommand{\one} {\ensuremath{\mathds{1} }} 
\newcommand{\T}{\ensuremath{\top}}

\newcommand{\funcRdR}{\ensuremath{{f}}}
\newcommand{\funcRdRd}{\ensuremath{{f}}}

\DeclareMathOperator*{\bigtimes}{\raisebox{-0.3ex}{\text{\Large$\times$}}}

\let\leq\leqslant
\let\geq\geqslant
\let\emptyset\varnothing

\newcommand{\calC}{\ensuremath{\mathcal{C}}}

\newcommand{\calE}{\ensuremath{\mathcal{E}}}
\newcommand{\calF}{\ensuremath{\mathcal{F}}}
\newcommand{\calG}{\ensuremath{\mathcal{G}}}

\newcommand{\calI}{\ensuremath{\mathcal{I}}}

\newcommand{\calL}{\ensuremath{\mathcal{L}}}

\newcommand{\calR}{\ensuremath{\mathcal{R}}}
\newcommand{\calS}{\ensuremath{\mathcal{S}}}
\newcommand{\calT}{\ensuremath{\mathcal{T}}}

\newcommand{\calX}{\ensuremath{\mathcal{X}}}

\newcommand{\frakso}{\ensuremath{\mathfrak{so}}}

\newcommand{\bmat}{\begin{matrix}}
\newcommand{\emat}{\end{matrix}}
\newcommand{\bbm}{\begin{bmatrix}}
\newcommand{\ebm}{\end{bmatrix}}
\newcommand{\bpm}{\begin{pmatrix}}
\newcommand{\epm}{\end{pmatrix}}
\newcommand{\bse}{\begin{subequations}}
\newcommand{\ese}{\end{subequations}}
\newcommand{\beq}{\begin{equation}}
\newcommand{\eeq}{\end{equation}}
\newcommand{\ben}{\begin{enumerate}}
\newcommand{\een}{\end{enumerate}}
\newcommand{\beni}{\renewcommand{\labelenumi}{\roman{enumi}.}
\renewcommand{\theenumi}{\roman{enumi}}\begin{enumerate}}
\newcommand{\eeni}{\end{enumerate}\renewcommand{\labelenumi}{\arabic{enumi}.}
\renewcommand{\theenumi}{\arabic{enumi}}}
\newcommand{\bena}{\renewcommand{\labelenumi}{\alpha{enumi}.}
\renewcommand{\theenumi}{\alpha{enumi}}\begin{enumerate}}
\newcommand{\eena}{\end{enumerate}\renewcommand{\labelenumi}{\arabic{enumi}.}
\renewcommand{\theenumi}{\arabic{enumi}}}
\newcommand{\bit}{\begin{itemize}}
\newcommand{\eit}{\end{itemize}}

\newcommand{\R}{\ensuremath{\mathbb R}}

\title{\LARGE \bf
Finite-time attitude synchronization \\ with distributed discontinuous protocols}

\author{Jieqiang Wei, Silun Zhang, Antonio Adaldo, Johan Thunberg, Xiaoming Hu, Karl H. Johansson  
\thanks{*This work is supported by Knut and Alice Wallenberg Foundation, Swedish Research Council, and Swedish Foundation for Strategic Research.}
\thanks{J. Wei, A. Adaldo and K.H. Johansson are with the ACCESS Linnaeus Centre, School of Electrical Engineering. 
 S. Zhang and X. Hu are with School of engineering sciences. 
 KTH Royal Institute of Technology,
 SE-100 44 Stockholm, Sweden.
 {\tt\small \{jieqiang, silunz, adaldo, hu, kallej\}@kth.se}.
 Johan Thunberg is with Luxembourg Centre for Systems Biomedicine, Université du Luxembourg,6, avenue du Swing, L-4367 Belvaux.
 {\tt\small \{johan.thunberg\}@uni.lu}
}}

\begin{document}

\maketitle

\begin{abstract}\label{s:Abstract}
The finite-time attitude synchronization problem is considered in this paper, where the rotation of each rigid body is expressed using the axis-angle representation. Two discontinuous and distributed controllers using the vectorized signum function are proposed, which guarantee almost global and local convergence, respectively. Filippov solutions and non-smooth analysis techniques are adopted to handle the discontinuities. Sufficient conditions are provided to guarantee finite-time convergence and boundedness of the solutions. Simulation examples are provided to verify the performances of the control protocols designed in this paper.
\end{abstract}


\IEEEpeerreviewmaketitle

\section{Introduction}\label{s:Introduction}

Originally motivated by aerospace developments in the middle of the last century~\cite{Bower1964,Kowalik1970}, the rigid body attitude control problem has continued to attract attention with many applications such as aircraft attitude control \cite{Athanasopoulos2014,Tsiotras1994}, spacial grabbing technology of manipulators~\cite{ZXLi}, target surveillance by unmanned vehicles~\cite{pettersen1996position}, and camera calibration in computer vision~\cite{ma2012invitation}. Furthermore, the configuration space of rigid-body attitudes is the compact non-Euclidean manifold $SO(3)$, which poses theoretical challenges for attitude control \cite{Bhat00scl}.

Here we review some related existing work. As attitude systems evolves on $SO(3)$---a compact manifold without a boundary---there exists no continuous control law that achieves global asymptotic stability \cite{brockett83asymptotic}. Hence one has to resort to some hybrid or discontinuous approaches. In \cite{lee2015global}, exponential stability is guaranteed for the tracking problem for a single attitude. The coordination of multiple attitudes is of high interest both in academic and industrial research, e.g., \cite{dong2016attitude,sarlette2009autonomous,Thunberg2016auto}.
In \cite{li2012decentralized} the authors considered the synchronization problem of attitudes under a leader-follower architecture. In \cite{pereira2016common}, the authors provided a local result on attitude synchronization. Based on a passivity approach, \cite{ren2010distributed} proposed a consensus control protocol for multiple rigid bodies with attitudes represented by modified Rodrigues parameters. In \cite{tron2012intrinsic}, the authors provided a control protocol in discrete time that achieves almost global synchronization, but it requires  global knowledge of the graph topology.
Although there exists no continuous control law that achieves global asymptotic stability, a methodology based on the axis-angle representation obtains almost global stability for attitude synchronization under directed and switching interconnection topologies is proposed in \cite{Thunberg2014auto}. These control laws were later generalized to include various types of vector representations including the Rodrigues Parameters and Unit Quaternions \cite{Thunberg2016auto}.
Besides these agreement results, \cite{lee2012relative,WJ15ac,wu2013spacecraft} provided distributed schemes for more general formation control of attitude in space $SO(3)$.

Among all the properties of attitude synchronization schemes, the finite-time convergence is an important one, because in practice it is desired that the system reaches the target configuration within a certain time-interval; consider, for instance, satellites in space that shall face a certain direction 
as they move in their orbits, or cameras that shall reach a certain formation to quickly follow an object.
So far,  finite-time attitude control problems are studied in different settings, e.g., \cite{Du2011,Zong2016}. In \cite{Du2011}, finite-time attitude synchronization was investigated in a leader-follower architecture, namely all the followers tracking the attitude of the leader. In \cite{Zong2016}, quaternion representation was employed for finite-time attitude synchronization. Both works used continuous control protocols with high-gain.


In this paper, we shall focus on the finite-time attitude synchronization problem, based on the axis-angle representations of the rotations without a leader-follower architecture, using discontinuous control laws. Two intuitive control schemes are proposed. The first scheme employs a direction-preserving sign function to guarantee finite-time synchronization almost globally, namely, the convergence holds for almost all the initial conditions. The other scheme, motivated by binary controllers for scalar multi-agent systems, e.g., \cite{Chen2011,LiuLam2016,Cortes2006,Hui2010}, uses the component-wise sign function. Compared to the first scheme, the second one is more coarse, in the sense that only finite number of control outputs are employed, and guarantees finite-time convergence locally. Since these control schemes are discontinuous,  nonsmooth analysis is employed throughout the paper.

The structure of the paper is as follows. In Section \ref{s:Preliminaries}, we review some results for the axis-angle representation for rotations in $SO(3)$ and introduce some terminologies and notations in the context of graph theory and discontinuous dynamical systems. Section \ref{ss:basic_model} presents the problem formulation of the finite-time attitude synchronization problem. The main results of the stability analysis of the finite-time convergence are presented in Section~\ref{s:main}, where an almost global and local stability are given in Subsection~\ref{ss:main-global} and \ref{ss:main-local} respectively. Then, in Section \ref{s:conclusion}, the paper is concluded.

\textbf{Notation.} With $\R_-,\R_+, \R_{\geq 0}$ and $\R_{\leqslant 0}$ we denote the sets of negative, positive, non-negative, non-positive  real numbers, respectively. The rotation group $SO(3)=\{R\in\R^{3\times 3}: RR^\top = I, \det R= 1 \}.$ The vector space of real $n$ by $n$ skew symmetric matrices is denoted as $\frakso(3)$. $\|\cdot\|_p$ denotes the $\ell_p$-norm and  the $\ell_2$-norm is denoted simply as $\|\cdot\|$ without a subscript.



\section{Preliminaries}\label{s:Preliminaries}

In this section, we briefly review some essentials about rigid body attitudes \cite{schaub},  graph theory \cite{Bollobas98}, and give some definitions for Filippov solutions \cite{filippov1988}.

Next lemma follows from Euler's Rotation Theorem.
\begin{lemma}
The exponential map
\begin{equation}
\exp :\frakso(3)\rightarrow SO(3)
\end{equation}
is surjective.
\end{lemma}

For any $p=[p_1,p_2,p_3]^\top\in\R^3$ and $\hat{p}\in\frakso(3)$ given as
\begin{equation}\label{e:basic:hat}
 \hat{p}:=
 \begin{pmatrix}
   0 &-p_3 &p_2\\
   p_3& 0 &-p_1\\
   -p_2& p_1 & 0
 \end{pmatrix},
\end{equation}
Rodrigues' formula is the right-hand side of
\begin{equation}\label{e:exp-map}
\exp(\hat{p}) = I_3+\frac{\sin(\theta)}{\theta}\hat{p}+\frac{1-\cos(\theta)}{\theta^2}(\hat{p})^2,
\end{equation}
where $\theta=\|p\|$ and $\exp(\hat{p})$ is the rotation matrix through an angle $\theta$ anticlockwise about the axis
$p$. For $R\in SO(3)$ where $R$ is not symmetric, we define the inverse of $\exp$ as
\begin{equation}\label{e:log-map}
\log(R) = \frac{\theta}{2\sin(\theta)}(R-R^\top),
\end{equation}
where $\theta=\arccos(\frac{\trace(R)-1}{2})$. We define $\log(I_3)$ as the zero matrix in $\R^{3\times 3}$. Note that \eqref{e:log-map} is not defined for $\theta=\pi$. The Riemannian metric for $SO(3)$ is defined as $d_R(R_1, R_2)=\frac{1}{\sqrt{2}}\|\log(R^{-1}_1R_2)\|_F$ where $\|\cdot\|_F$ is the Frobenius norm.


One important relation between $SO(3)$ and $\R^3$ is that the open ball $B_\pi(I)$ in $SO(3)$ with radius $\pi$ around the identity, which is almost the whole $SO(3)$, is diffeomorphic to the open ball $B_{\pi}(0)$ in $\R^3$ via the logarithmic and the exponential map defined in \eqref{e:log-map} and \eqref{e:exp-map}.

\medskip

An undirected \emph{graph} $\calG=(\calI,\calE)$ consists of a finite set of nodes $\calI = \{1,2,\ldots,n\}$ and a set of edges $\calE\in\calI\times\calI$ of unordered pairs of $\calI$. To each edge $(i,j)\in\calE$, we associate a weight $w_{ij}>0$. The weighted adjacency matrix $A = [a_{ij}]\in\R^{n\times n}$ is defined by $a_{ij} = a_{ji} = w_{ij}$ if $(i,j)\in\calE$ and $a_{ji} = 0$ otherwise. Note that $A = A^{\T}$ and that $a_{ii}=0$ as no self-loops are allowed. For each node $i\in\calI$, its degree $d_i$ is defined as $d_i = \sum_{j=1}^n a_{ij}$. The graph Laplacian $L$ is defined as $L = \Delta - A$ where $\Delta$ is a diagonal matrix such that $\Delta_{ii}=d_i$. As a result, $L\one = \mathbf{0}$. We denote the set of neighbors of node $i$ as $N_i=\{j\in\calI \mid w_{ij}>0 \}$.
If the edges are ordered pairs of $\calI$, the graph $\calG$ is called a \emph{directed graph}, or \emph{digraph} for short.
An edge of a digraph $\mathcal{G}$ is denoted by $(i,j)$ (with $i\neq j$) representing the tail vertex $i$ and the head
vertex $j$ of this edge.
A digraph with unit weights is completely specified by its \emph{incidence
matrix} $B\in\R^{n\times m}$, where $|\calE|=m$, with $B_{ij}$
equal to $-1$ if the $j$th edge is towards vertex
$i$, and equal to $1$ if the $j$th edge is originating from
vertex $i$, and $0$ otherwise. The incidence matrix for undirected graphs is defined by adding arbitrary orientations to the edges of the graph.
Finally, we say that a graph $\calG$ is connected if, for any two nodes $i$ and $j$, there exists a sequence of edges that connects them. In order to simplify the notation in the proofs, we set the weight $w_{ij}$ to be one. All the results in this paper hold for the general case where the $w_{ij}$'s are elements in $R_+$.


\medskip

In the remainder of this section, we discuss Filippov solutions. Let $\funcRdRd$ be a map from $\R^m$ to $\R^n$ and let $2^{\R^n}$ denote the collection of all subsets of $\R^n$. The \emph{Filippov set-valued map} of $\funcRdRd$, denoted $\calF[\funcRdRd]:\R^m\rightarrow 2^{\R^n}$, is defined as
\begin{equation*}
\calF[\funcRdRd](x) := \bigcap_{\delta>0}\bigcap_{\mu(S)=0}\overline{\mathrm{co}}\big\{ \funcRdRd(B(x,\delta)\backslash S) \big\},
\label{eqn_Filippovdef}
\end{equation*}
where $S$ is a subset of $\R^m$, $\mu$ denotes the Lebesgue measure, $B(x,\delta)$ is the ball centered at $x$ with radius $\delta$ and $\overline{\mathrm{co}}\{\calX\}$ denotes the convex closure of a set $\calX$. If $\funcRdRd$ is continuous at $x$, then $ \calF[\funcRdRd](x)$ contains only the point $\funcRdRd(x)$.


A \emph{Filippov solution} of the differential equation $\dot{x}(t)=\funcRdRd(x(t))$ on $[0,T]\subset\R$ is an absolutely continuous function $x:[0,T]\rightarrow\R^n$ that satisfies the differential inclusion
\begin{equation}\label{e:differential_inclusion}
\dot{x}(t)\in \calF[\funcRdRd](x(t))
\end{equation}
for almost all $t\in[0,T]$. A Filippov solution is \emph{maximal} if it cannot be extended forward in time, that is, if it is not the result of the truncation of another solution with a larger interval of definition. Next, we introduce invariant sets, which will play a key part further on. Since Filippov solutions are not necessarily unique, we need to specify two types of invariant sets. A set $\calR\subset\R^n$ is called \emph{weakly invariant} if, for each $x_0\in \calR$, at least one maximal solution of \eqref{e:differential_inclusion} with initial condition $x_0$ is contained in $\calR$. Similarly, $\calR\subset \R^n$ is called \emph{strongly invariant} if, for each $x_0\in \calR$, every maximal solution of \eqref{e:differential_inclusion} with initial condition $x_0$ is contained in $\calR$. For more details, see \cite{cortes2008, filippov1988}. We use the same definition of regular function as in \cite{Clarke1990optimization} and recall that any convex function is regular.


For $V:\R^n\rightarrow\R$ locally Lipschitz, the \emph{generalized gradient} $\partial V:\R^n\rightarrow 2^{\R^n}$ is defined by
\begin{equation}
\partial V(x):=\mathrm{co}\Big\{\lim_{i\rightarrow\infty} \nabla
V(x_i)\mid x_i\rightarrow x, x_i\notin S\cup \Omega_{\funcRdR} \Big\},
\end{equation}
where $\nabla$ is the gradient operator, $\Omega_{\funcRdR} \subset\R^n$ is the set of points where $V$ fails to be differentiable and $S\subset\R^n$ is a set of measure zero that can be
arbitrarily chosen to simplify the computation, since the resulting set $\partial V(x)$ is independent of the choice of $S$ \cite{Clarke1990optimization}.

Given a  set-valued map $\calT:\R^n\rightarrow 2^{\R^n}$, the \emph{set-valued Lie derivative} $\calL_{\calT}V:\R^n\rightarrow 2^{\R^n}$ of a locally Lipschitz function $V:\R^n\rightarrow \R$  with respect to $\calT$ at $x$ is defined as
\begin{equation}\label{e:set-valuedLie}
\begin{aligned}
\calL_{\calT}V(x) := & \big\{ a\in\R \mid \exists\nu\in\calT(x) \textnormal{ such that } \\
& \quad \zeta^T\nu=a,\, \forall \zeta\in \partial V(x)\big\}.
\end{aligned}
\end{equation}
If $\calT(x)$ is convex and compact $\forall x\in\R^n$, then $\calL_{\calT}V(x)$ is a compact interval in $\R$, possibly empty.



The $i$th row of a matrix $M$ is denoted as $M_{i}$.
For any matrix $M$, we denote $M\otimes I_3$ as $\hat{M}$ and $M_{i}\otimes I$ as $\hat{M}_i$. A positive definite and semidefinite (symmetric) matrix $M$ is denoted as $M> 0$ and $M\geq 0$, respectively.  The vectors $e_1,e_2,\ldots,e_n$ denote the canonical basis of $\R^n$. 
The vectors $\one_n$ and $\mathbf{0}_n$ represents a $n$-dimensional column vector with each entry being $1$ and $0$, respectively. In this paper, we define the direction-preserving signum as
\begin{equation}\label{e:sign}
\sign(w) = \begin{cases}
\frac{w}{\|w\|} & \textrm{ if } w\neq \mathbf{0},\\
0 & \textrm{ if } w=\mathbf{0},
\end{cases}
\end{equation}
for $w\in\R^k$. The component-wise signum is denoted as
\begin{equation}\label{e:signc}
\sign_c(w)=[\sign(w_1),\ldots,\sign(w_k)]^\top,
\end{equation}
where $w=[w_1,w_2,\ldots,w_k]^\top$.
Notice that for scalars, these two signum functions coincide. Furthermore, component-wise signum is coarser than direction-preserving in the sense that there is only a finite number of elements in the range for a fixed dimension $k$.

\section{Problem formulation}\label{ss:basic_model}

We consider a system of $n$ agents (rigid bodies). We denote the world frame as $\calF_w$ and the instantaneous body frame of agent $i$ as $\calF_i$ where $i\in\calI=\{1,\ldots,n\}$. Let $R_i(t)\in SO(3)$ be the attitude of $\calF_i$ relative to $\mathcal{F}_w$ at time $t$,
 and, when $R_i(t)\in B_\pi(I)$, the corresponding axis-angle representation $x_i(t)\in\R^3$ be given by
\begin{equation}\label{e:basic:log}
  \hat{x}_i(t)=\log (R_i(t)).
\end{equation}
The kinematics of $x_i$ is given by \cite{schaub}
\begin{equation}\label{e:basic_model}
\dot{x}_i=L_{x_i}\omega_i, \quad i\in\calI
\end{equation}
where $\omega_i$ is the control signal corresponding to the instantaneous angular velocity of $\calF_i$ relative to $\mathcal{F}_w$ expressed in the frame $\mathcal{F}_i$, and the transition matrix $L_{x_i}$ is defined as
\begin{align}
L_{x_i}  =& I_3+\frac{\hat{x}_i}{2}+\Bigg( 1-\frac{\sinc(\|x_i\|)}{\sinc^2(\frac{\|x_i\|}{2})} \Bigg) \Big( \frac{\hat{x}_i}{\|x_i\|} \Big)^2 \nonumber\\
 =& \frac{\sinc(\|x_i\|)}{\sinc^2(\frac{\|x_i\|}{2})} I_3+\Bigg(1-\frac{\sinc(\|x_i\|)}{\sinc^2(\frac{\|x_i\|}{2})}\Bigg) \frac{x_i x_i^\top}{\|x_i\|^2}+\frac{\hat{x}_i}{2} \nonumber\\
=:& L^1_{x_i}+\frac{\hat{x}_i}{2},
\end{align}
where $\sinc(\alpha)$ is defined as $\alpha \sinc(\alpha) = \sin(\alpha)$ for all $\alpha\neq 0$ and $\sinc(0) = 1$, see \cite{schaub}. Note that for $\|x_i\|\in [0,\pi]$, the function $\frac{\sinc(\|x_i\|)}{\sinc^2(\frac{\|x_i\|}{2})}$ is concave and belongs to $[0,1]$. Then the symmetric part of $L_{x_i}$, denoted by $L^1_{x_i}$, is positive semidefinite. More precisely, $L^1_{x_i}>0$ if $\|x_i\|\in [0,2\pi)$. Moreover, $L_{x_i}$ is Lipschitz on $B_r(0)$ for any $0<r<\pi$ (see \cite{Thunberg2014auto}).

The system \eqref{e:basic_model} can be written in a compact form as
\begin{equation}\label{e:basic_model_comp}
\dot{x}=L_{x}\omega
\end{equation}
where
\begin{equation}
\begin{aligned}
x & =[x^\top_1,\ldots,x^\top_n]^\top, \\
L_x & =\blcdiag{(L_{x_1},\ldots, L_{x_n})}, \\
\omega & =[\omega^\top_1, \ldots,\omega^\top_n]^\top.
\end{aligned}
\end{equation}

For the multi-agent system \eqref{e:basic_model_comp}, we assume that the agents can communicate with each other about the state variables $x_i$ via an undirected connected graph $\calG$. The aim is to design control protocols for $\omega$ such that the absolute rotations of all agents converge to a common rotation in the world frame $\calF_w$ in finite time, i.e.,
\begin{align}
\exists T>0, \bar{R} \in SO(3) \textnormal{ s.t. } R_i \rightarrow \bar{R} , \forall i\in\calI,  \textnormal{ as } t\rightarrow T.
\end{align}
This is equivalent to that $x$ converges to the consensus space
\begin{align}
\calC=\{x\in\R^{3n} \mid \exists \bar{x}\in\R^3 \textnormal{ such that } x=\mathds{1}_n\otimes \bar{x} \}
\end{align}
in finite time. We shall propose two distributed controllers that achieve this goal.

\section{Main result}\label{s:main}

In this section, we shall first present a control law that guarantees that the rotations of all the rigid bodies converge to a common rotation for any initial condition $R_i(0)\in B_\pi(I)\subset SO(3)$ for all $i\in\calI$. Note that this initial condition in $SO(3)$ is equivalent to $\|x_i(0)\|<\pi$ under the axis-angle representation.
In order to avoid the singularity of the logarithmic map \eqref{e:log-map}, the control law makes sure that the constraint $\|x_i(t)\|<\pi$ is met for all $i\in\calI$ and for all time $t>0$.  We consider controllers of the following form
\begin{equation}\label{e:controller1}
\omega_i=f_i~ \big(\sum_{j\in N_i}(x_j-x_i)\big), \quad i\in\calI,
\end{equation}
with maps $f_i:\R^3\rightarrow\R^3$ and the elements in the set $N_i$ are the neighbors of agent $i$. Now the closed-loop system can be written in a compact form as
\begin{equation}\label{e:closedloop1}
\dot{x} =L_{x}f\big( -\hat{L}x \big) \\
\end{equation}
where $f(y)=[f^\top_1(y_1),\ldots,f^\top_n(y_n)]^\top$, $\hat{L}=L\otimes I_3$ and $L$ is the Laplacian of the graph.
%
Our control design is based on the signum function. 
More precisely, we consider the case when some of the functions $f_i$ are $\sign$ or $\sign_c$, while the others satisfy certain continuity assumptions to be defined in the following subsections. We propose two control protocols which guarantee almost global, in the sense of  $R_i(0)\in B_\pi(I)\subset SO(3)$, and local convergence, respectively.  As discontinuities are introduced into  \eqref{e:closedloop1} by the signum functions, we shall understand the trajectories in the sense of Filippov, namely an absolutely continuous function $x(t)$ satisfying the differential inclusion
\begin{equation}\label{e:inclusion1}
\begin{aligned}
\dot{x} & \in \calF[L_{x}f\big( -\hat{L}x \big)](x) \\
& = L_{x}\calF[f\big( -\hat{L}x \big)](x) \\
& =: \calF_1(x)
\end{aligned}
\end{equation}
for almost all time, where we used Theorem 1(5) in \cite{paden1987}.

\subsection{Control law for global convergence}\label{ss:main-global}

In this subsection, we shall design a controller such that finite-time synchronization is achieved for any initial condition $R_i(0)\in B_\pi(I)\subset SO(3)$ by using the direction preserving $\sign$ defined in \eqref{e:sign}.
It might seem natural to let $f_i=\sign$ for all $i\in\calI$.
However, the following example shows that this simple controller does not guarantee $\|x_i(t)\|<\pi, \forall t>0$ for all Filippov solutions.

\begin{example}\label{ex:sliding}
Consider the system
\begin{equation*}
\begin{aligned}
\dot{x}_1 & = L_{x_1}\sign(x_2+x_3-2x_1)\\
\dot{x}_2 & = L_{x_2}\sign(x_1+x_3-2x_2)\\
\dot{x}_3 & = L_{x_3}\sign(x_2+x_1-2x_3)\\
\end{aligned}
\end{equation*}
defined on a complete graph. 
We show that for $t_0$ such that $x(t_0)\in\calC$, the trajectories can violate the constraints $\|x_i(t)\|<\pi, i\in\calI$, for some $t>t_0$.

First, by Theorem 1(1) in \cite{paden1987}, we have for any $x\in\calC$, there exists an $\varepsilon$, independent of $x$, such that the ball $B_\varepsilon(0)\subset\calF[\sign(-\hat{L}x)](x)\subset\R^9$. Second, suppose $x(t_0)=\mathds{1}\otimes\bar{x}$ for some $\bar{x}\in\R^3$. Then there exists $0<\varepsilon_1<\varepsilon$ such that the vector $\varepsilon_1\mathds{1}\otimes\frac{\bar{x}}{\|\bar{x}\|}\in B_\varepsilon(0)$. Hence,
\begin{align*}
x(t)=\mathds{1}\otimes ((t-t_0)\varepsilon_1\frac{\bar{x}}{\|\bar{x}\|}+\bar{x}), \ t\geq t_0
\end{align*}
is a Filippov solution. Indeed,
\begin{align*}
\dot{x}(t) & = \varepsilon_1\mathds{1}\otimes\frac{\bar{x}}{\|\bar{x}\|} \\
& =  L_{x(t)}\varepsilon_1\mathds{1}\otimes\frac{\bar{x}}{\|\bar{x}\|} \\
& \in L_{x(t)}\calF[\sign(-\hat{L}x)](x(t))
\end{align*}
where the second equality follows from 
\begin{align*}
L_{x(t)}\varepsilon_1\mathds{1}\otimes\frac{\bar{x}}{\|\bar{x}\|} & = L_{x(t)} x(t)\frac{\varepsilon_1}{\|\bar{x}\|+(t-t_0)\varepsilon_1} \\
& = x(t)\frac{\varepsilon_1}{\|\bar{x}\|+(t-t_0)\varepsilon_1} \\
& = \varepsilon_1\mathds{1}\otimes\frac{\bar{x}}{\|\bar{x}\|}.
\end{align*}
Here we used the fact that $L_{x(t)} x(t)=x(t)$. Then for large enough $t$, $\|x_i(t)\|$ can be larger than $\pi$. The solutions of the type $\mathds{1}\otimes \eta(t)$ with $\eta(t)$ a non-constant function is called \emph{sliding consensus}.

\end{example}

The previous example 
motivates us to consider the following assumption.

\begin{assumption}\label{ass_vectorfield_1}
For some set $\calI_c\subset\calI$, the function $f$ in~(\ref{e:closedloop1}) satisfies the following conditions:
	\begin{enumerate}
		\item[\textit{(i)}] For $i\in\calI_c$,  $f_i:\R^k\rightarrow\R^k$ is locally Lipschitz continuous and satisfies $f_i(\mathbf{0}) = \mathbf{0}$ and $f_i(y)^\top y=\|f_i(y)\|\|y\|\neq 0$ for all $y\neq \mathbf{0}$;
		\item[\textit{(ii)}] For $i\in\calI\backslash\calI_c$, the function $f_i = \sign$.
	\end{enumerate}
\end{assumption}

Note that Condition \textit{(i)} in Assumption \ref{ass_vectorfield_1} corresponds to that $f_i$ is direction preserving. 

Before showing the result for finite-time convergence, we formulate a condition for the controller \eqref{e:controller1} satisfying Assumption \ref{ass_vectorfield_1} such that the set $\{ x \mid \|x_i\|<\pi \}$ is strongly invariant for the dynamics \eqref{e:closedloop1}.

\begin{lemma}\label{lm:invariant1}
Consider the differential inclusion \eqref{e:inclusion1} satisfying Assumption \ref{ass_vectorfield_1}. If one of the following two conditions is satisfied
\begin{enumerate}
\item[\textit{(i)}] $|\calI| = 2$ and $|\calI_c| = 0$;
\item[\textit{(ii)}] $|\calI| \geq 2$ and $|\calI_c| \geq 1$,
\end{enumerate}
then the set $\calS_1(C) := \{ x\in\R^{3n}\mid \|x_i\|\leq C, i\in\calI \}$, where $C<\pi$ is a constant, is strongly invariant. This implies that $B_\pi(I)^n$ is strongly invariant.
\end{lemma}

\begin{proof}
We use a Lyapunov-like argument to prove that for any initial condition in $\calS_1(C)$, all the solutions of \eqref{e:inclusion1} will remain within the set.

Consider the Lyapunov function candidate $V(x)=\max_{i\in\calI} \|x_i\|^2$. Notice that $V$ is convex, hence regular.
Let
\begin{align}
\alpha(x) = \big\{ i\in\calI \mid \|x_i\|_2^2 = V(x) \big\}.
\label{eqn_prf_lm:invariant_alpha}
\end{align}	
The generalized gradient of $V$ is given as
\begin{equation}
\partial V(x)=\mathrm{co} \{e_i\otimes x_i \mid i\in\alpha(x) \}.
\end{equation}

Next, let $\Psi$ be defined as
\begin{equation}
\Psi = \big\{ t\geq 0 \mid \textnormal{both } \dot{x}(t) \textnormal{ and } \tfrac{d}{dt}V(x(t)) \textnormal{ exist} \big\}.
\end{equation}
Since $x$ is absolutely continuous (by definition of Filippov solutions) and $V$ is locally Lipschitz, by Lemma~1 in~\cite{Bacciotti1999} it follows that $\Psi=\R_{\geq 0}\setminus\bar{\Psi}$ for a set $\bar{\Psi}$ of measure zero and that
\begin{equation}
\frac{d}{dt}V(x(t))\in \calL_{\calF_1}V(x(t))
\end{equation}
for all $t\in\Psi$, so that the set $\calL_{\calF_1}V(x(t))$ is nonempty. For $t\in\bar{\Psi}$, we have that $\calL_{\calF_1}V(x(t))$ is empty, and hence $\max \calL_{\calF_1}V(x(t)) = -\infty < 0$ by definition. Therefore, we only consider $t\in\Psi$ in the rest of the proof.

By using Theorem 1(4) and (5) in \cite{paden1987}, the differential inclusion can be enlarged as follows
\begin{equation}\label{e:inclusion1_bigger}
\begin{aligned}
\dot{x} & \in\calF_1(x) \\
& \subset \bigtimes_{i=1}^n \calF[L_{x_i}f_i(-\hat{L}_ix)](x) \\
& =  \bigtimes_{i=1}^n L_{x_i}\calF[f_i(-\hat{L}_ix)](x) \\
& =  \bigtimes_{i=1}^n L_{x_i}\calF[f_i](-\hat{L}_ix) \\
& =: \calF_2(x),
\end{aligned}
\end{equation}
where the first equality follows from Assumption \ref{ass_vectorfield_1} and the fact that $L_{x_i}$ is continuous for $x_i$ with $\|x_i\|<2\pi$, thus we can use Theorem 1(5) in \cite{paden1987}. Moreover, we obtain that $\calL_{\calF_1}V(x(t))\subset \calL_{\calF_2}V(x(t))$ for all $t\geq 0$. For the rest of the proof, we shall show $\calL_{\calF_2}V(x(t))\subset \R_{\leq 0}$ by considering two cases.


\noindent
\textbf{Case 1: }For $x\in\calC$, i.e., $\alpha(x) = \calI$, the following two subcases can be distinguished.
\begin{enumerate}
	\item[\textit{(i)}] $|\calI| \geq 2$ and $|\calI_c| \geq 1$. There is $i\in\calI$ such that $f_i$ is locally Lipschitz and direction preserving. Then, by using the definition of the Filippov set-valued map, one can show that $\nu_i = \mathbf{0}_3$ for all $\nu=[\nu_1,\ldots,\nu_n]\in\calF_2(x)$ (recall that $x\in\calC$). As $\calL_{\calF_2}V(x(t))$ is nonempty (by considering $t\in\Psi$), there exists $a\in\calL_{\calF_2}V(x(t))$ such that $a = \zeta^\top\nu$ for all $\zeta\in\partial V(x(t))$, see the definition (\ref{e:set-valuedLie}). By choosing $\zeta = e_i\otimes x_i(t)$, it follows that $a = (e_i\otimes x_i)^\top\nu_i = 0$, which implies that $\max\calL_{\calF_2}V(x(t)) \leq 0$. 

	\item[\textit{(ii)}] $|\calI| = 2$ and $|\calI_c| = 0$. In the following, we consider the Filippov solution of system \eqref{e:inclusion1}, which can be written as
	\begin{equation}
	\begin{aligned}
	\dot{x}_1 & = L_{x_1}\frac{x_2-x_1}{\|x_2-x_1\|}, \\
	\dot{x}_2 & = L_{x_2}\frac{x_1-x_2}{\|x_1-x_2\|}.
	\end{aligned}\label{eqn_prf_lm:invariant_twoagents}
	\end{equation}
	Then it can be shown that, for $x_1 = x_2$ (i.e., $x\in\calC$), any element $\nu$ in the Filippov set-valued map of \eqref{eqn_prf_lm:invariant_twoagents} satisfies $\nu_1 = -\nu_2$. Stated differently, the following implication holds for $\nu = [\nu_1^\top,\nu_2^\top]^\top$:
	\begin{align}
	\nu\in\calF[h](x),\; x\in\calC \;\Rightarrow\; \nu_1 = -\nu_2.
	\end{align}
	Next, by recalling that $\alpha(x) = \calI$, it follows that
	\begin{align}
	\partial V(x) = \mathrm{co}\big\{e_1\otimes x_1, e_2\otimes x_2\big\}
	\end{align}
	with $x_1 = x_2$. Now, following a similar reasoning as in item \textit{(i)} on the basis of the definition of the set-valued Lie derivative in (\ref{e:set-valuedLie}), it can be concluded that $a = \zeta^\top\nu=0$, so that $\max\calL_{\calF_1}V(x(t)) = 0$ for all $x\in\calC$.
\end{enumerate}

\noindent
\textbf{Case 2:}
For $x\notin\calC$,
take an index $i\in\alpha(x)$ such that $\hat{L}_{i}x \neq \mathbf{0}_3$. Note that such $i$ indeed exists. Namely, assume in order to establish a contradiction that $\hat{L}_{i}x = \mathbf{0}_3$ for all $i\in\alpha(x)$. Then, it holds that
\begin{equation}
0 = x_i^\top\hat{L}_{i}x = \sum_{j\in N_i} x_i^\top (x_i - x_j).
\label{eqn_prf_lm:invariant_Lix}
\end{equation}
Since $\|x_i\|\geq\|x_j\|$, it follows from (\ref{eqn_prf_lm:invariant_Lix}) that $x_j = x_i$ for all $j\in N_i$. By repeating this argument and recalling that the interconnection topology is connected, it follows that $x_j = x_i$ for all $j\in\calI$, i.e., $x\in\calC$. This is a contradiction to $x\notin\calC$.

For the index $i\in\alpha(x)$ satisfying $\hat{L}_ix\neq \mathbf{0}_3$, it follows from Assumption~\ref{ass_vectorfield_1} that there exists $\gamma_i>0$ such that
\begin{align}
\calF[f_i]\big(-\hat{L}_ix\big) = -\gamma_i\hat{L}_ix,
\end{align}
i.e., for any $\nu\in\calF_2(x)$ it holds that $\nu_i = -\gamma_i L_{x_i}\hat{L}_ix$.
Note that this is a result of the direction-preserving property of either the vectorized signum function (for nonzero argument, then $\gamma_i = 1$) or the Lipschitz continuous function (by Assumption~\ref{ass_vectorfield_1}). Then, choosing $\zeta\in\partial V(x)$ as $\zeta = e_i\otimes x_i$ (recall that $i\in\alpha(x)$), it follows that for any $\nu\in\calF_2(x)$ we have
\begin{align}
\zeta^\top \nu = & -\gamma_i x^\top_i L_{x_i}\hat{L}_ix \nonumber\\
= & -\gamma_i x^\top_i \hat{L}_ix \nonumber\\
< & 0,
\end{align}
where the second equality is based on $L_{x_i}x_i=x_i$.

Summarizing the results of the two cases leads to
\begin{align}
\max\calL_{\calF_2}V(x)\leq 0
\end{align}
if $\|x_i\|<2\pi$ for all $i\in\calI$. Since the trajectory $x(t)$ is absolutely continuous, we have that if $\|x_i(0)\|\leq C<\pi$ for all $i\in\calI$, all the trajectories remain within the set $\calS_1(C)$.
\end{proof}

\begin{remark}
As indicated in Example \ref{ex:sliding}, sliding consensus can happen when $\calI_c=\emptyset$ and $|\calI|>2$. This will violate the strong invariance of the set $\calS_1(C)$ with $C<\pi$, which will introduce singularity for the axis-angle representation for rotations.
\end{remark}

Before we prove the finite-time convergence, we provide a sufficient condition for that all Filippov solutions of \eqref{e:inclusion1} converge to consensus asymptotically.

\begin{lemma}\label{lm:asymptotic1}
Under the same assumptions of Lemma \ref{lm:invariant1}, all Filippov solutions of \eqref{e:inclusion1} asymptotically converge to static consensus.
\end{lemma}

\begin{proof}
Similar to the proof of Lemma \ref{lm:invariant1}, we shall prove that the conclusion holds for the bigger inclusion given by \eqref{e:inclusion1_bigger}. 
In this proof cases \textit{(i)} and \textit{(ii)} can be handled with the same arguments.

Consider the Lyapunov function candidate $V(x)=\sqrt{x^\top \hat{L} x}$, which is convex, hence regular. The generalized gradient of $V$ is given as follows:
\begin{equation}
\partial V(x) = \begin{cases}
\frac{\hat{L}x}{\sqrt{x^\top \hat{L} x}} & \textrm{ if } x\notin \calC,\\
\mathrm{co}\Big\{\lim_{y\rightarrow x} \frac{\hat{L}y}{\sqrt{y^\top \hat{L} y}}: y\notin \calC \Big\} & \textrm{ if } x\in\calC.
\end{cases}
\end{equation}

Next we shall calculate the Lie derivative of $V$ by considering two cases.
\begin{enumerate}
\item[\textit{(i)}] If $x\notin\calC$, the Lie derivative is given as
\begin{equation}\label{e:lie_controller1}
\begin{aligned}
\calL_{\calF_2}V(x) & = \frac{x^\top\hat{L}}{\sqrt{x^\top \hat{L} x}} \calF_2 \\
& = \frac{\sum_{i\in\calI} x^\top\hat{L}^\top_i L_{x_i} \calF[f_i](-\hat{L}_i x) }{\sqrt{x^\top \hat{L} x}}.
\end{aligned}
\end{equation}
Here we have $(\hat{L}_ix)^\top L_{x_i} \calF[f_i](\hat{L}_ix)\geq 0$ for $i\in\calI$. Indeed, it is because: (1), the conditions that the matrix $L^1_{x_i}>0$ for $x_i$ satisfying $\|x_i\|<\pi$; (2), Assumption \ref{ass_vectorfield_1} about direction preservation, and (3), the set $\calS_1(C)$ is strongly invariant for $C<\pi$.  Moreover, if $\hat{L}_ix\neq\mathbf{0}_3$, the set $(\hat{L}_ix)^\top L_{x_i} \calF[f_i](\hat{L}_ix)\subset \R_{>0}$. Hence $\calL_{\calF_2}V\subset\R_{< 0}$.

\item[\textit{(ii)}] If $x\in\calC$, it can be seen that $\zeta\in\partial V(x)$ implies $-\zeta\in\partial V(x)$. Hence if $\calL_{\calF_2}V\neq \emptyset$, it has to be $\{0\}$. In fact, by taking $\mathbf{0}\in\calF_2$, we have that $0\in\calL_{\calF_2}V$.
\end{enumerate}

Next, by Theorem 3 in \cite{Cortes2006}, it holds that all Filippov solutions of \eqref{e:inclusion1_bigger} converge to the set $\overline{Z_{\calF_2,V}}$ asymptotically. The remaining task is to characterize the set $\overline{Z_{\calF_2,V}}$. So far we have shown that $x\notin Z_{\calF_2,V}$  $\forall x\notin \calC$, which implies that $Z_{\calF_2,V}\subset\calC$. By the fact that $\calC$ is closed, we have $\overline{Z_{\calF_2,V}}\subset \calC$. Moreover, when $x\in\calC$, $\dot{x}_i=0$ where $i\in\calI_c$, which implies that $x_i$ remains constant. In conclusion, asymptotic convergence to static consensus is guaranteed.

\end{proof}

Now we are ready for the main result of this section.

\begin{theorem}\label{thm:main1}
Assume that $R_i(0)\in B_\pi(I)\subset SO(3)$ for all $i\in \calI$ and that the graph $\calG$ is connected. Consider the multi-agent system \eqref{e:closedloop1} satisfying Assumption~\ref{ass_vectorfield_1} and the corresponding differential inclusion \eqref{e:inclusion1}. Then, all Filippov solutions converge to consensus 
in finite time if one of the following conditions holds:
	\begin{itemize}
		\item[\textit{(i)}]  $|\calI| > 2$ and $|\calI_c| = 1$;
		\item[\textit{(ii)}]  $|\calI| = 2$ and $|\calI_c| \leq 1$.
	\end{itemize}
\end{theorem}

\begin{proof}
The proof is separated into two parts, one for each of the conditions.

\textit{(i)} Without loss of generality, we assume that $f_1$ satisfies the condition $\textit{(i)}$ in Assumption \ref{ass_vectorfield_1} while $f_2, \ldots,f_n$ are $\sign$. Similar to Lemma \ref{lm:invariant1}, instead of proving the conclusion for the differential inclusion \eqref{e:inclusion1}, we shall show that it holds for the bigger inclusion given by \eqref{e:inclusion1_bigger}.

Consider the Lyapunov function candidate $V(x)=\sqrt{x^\top \hat{L} x}$. We shall show that there exists $c$ such that $\max\calL_{\calF_2}V<c<0$ for any initial condition $\calS_1(C)\setminus \calC$ with $C<\pi$.

In the proof of Lemma \ref{lm:asymptotic1}, we have shown that for $x\notin \calC$, the Lie derivative is given by \eqref{e:lie_controller1}. By the fact that $L^1_{x_i}>0$ with $\|x_i\|<\pi$ and $f_i$ is direction preserving, we have
\begin{equation}
\calL_{\calF_2}V \leq \frac{\sum_{i=2}^n x^\top\hat{L}^\top_i L_{x_i} \calF[\sign](-\hat{L}_i x) }{\sqrt{x^\top \hat{L} x}}.
\end{equation}
Furthermore, the Filippov set-valued map
\begin{equation}
\calF[\sign](-\hat{L}_i x) =
\begin{cases}
\{\frac{-\hat{L}_i x}{\|\hat{L}_i x\|} \} & \textrm{ if } \|\hat{L}_i x\|\neq 0,\\
\{ v\mid \|v\|\leq 1 \} & \textrm{ if } \|\hat{L}_i x\|= 0,
\end{cases}
\end{equation}
which implies that
\begin{equation}
\begin{aligned}
& x^\top\hat{L}^\top_i L_{x_i}\calF[\sign](-\hat{L}_i x) \\
= &
\begin{cases}
\{\frac{-x^\top\hat{L}^\top_i L_{x_i}\hat{L}_i x}{\|\hat{L}_i x\|} \} & \textrm{ if } \|\hat{L}_i x\|\neq 0,\\
\{\|\hat{L}_i x\| \} & \textrm{ if } \|\hat{L}_i x\|= 0.
\end{cases}
\end{aligned}
\end{equation}
Note that, for any $x$ satisfying $\|x_i\|<\pi$ for all $i\in\calI$, there exists $c_1\in(0,1)$, which only depends on $\max_i \|x_i(0)\|$, such that $L_{x_i}-c_1I\geq 0$ for all $i\in\calI$. This implies that
\begin{equation}
x^\top\hat{L}^\top_i L_{x_i}\calF[\sign](-\hat{L}_i x)\subset (-\infty, -c_1\|\hat{L}_i x\|].
\end{equation}

So far we have shown that, for any $a\in\calL_{\calF_2}V(x)$, it holds that
\begin{equation}\label{eqn_prf_tm:main_vector_suff_abound_step1}
a\leq -c_1\frac{\sum_{i=2}^n \|\hat{L}_i x \|}{\sqrt{x^\top \hat{L} x}}.
\end{equation}
Furthermore, by using that $\hat{L}_1x=-\sum_{i=2}^n \hat{L}_i x$, which is based on the connectivity of the graph $\calG$, we have
\begin{align}
\|\hat{L}_1x\| = \left\| \sum_{i=2}^n \hat{L}_ix \right\| \leq \sum_{i=2}^n \big\| \hat{L}_ix \big\|,
\label{eqn_prf_tm:main_vector_suff_L1bound}
\end{align}
where the triangle inequality is used. Then, the use of \eqref{eqn_prf_tm:main_vector_suff_L1bound} in \eqref{eqn_prf_tm:main_vector_suff_abound_step1} yields
\begin{align}
a \leq -\frac{c_1}{2}\frac{1}{\sqrt{x^\top\hat{L}x}} \left( \sum_{i=1}^n \|L_ix\| \right).
\label{eqn_prf_tm:main_vector_suff_abound_step2}
\end{align}
By exploiting the observation that $L$ is a graph Laplacian, it holds that
\begin{align}
L = U^\top\Lambda U, \quad L^\top L = U^\top\Lambda^2 U,
\label{eqn_prf_tm:main_vector_suff_Ldecomposition}
\end{align}
where $\Lambda = \diag\{0,\lambda_2,\ldots,\lambda_n\}$ is a diagonal matrix with real-valued eigenvalues satisfying $0<\lambda_2$ and $\lambda_{j}\leq\lambda_{j+1}$ for $j=2,\ldots,n$. The matrix $U$ collects the corresponding right-eigenvectors. From (\ref{eqn_prf_tm:main_vector_suff_Ldecomposition}), it can be seen that
\begin{align}
L^\top L - cL \geq 0
\end{align}
for any $c\in[0,\lambda_2]$. Consequently, using $\hat{L} = L\otimes I$, it follows that
\begin{align}
\left( \sum_{i=1}^n \|\hat{L}_ix\| \right)^2 = x^\top \hat{L}^\top \hat{L}x \geq c x^\top \hat{L}x.
\label{eqn_prf_tm:main_vector_suff_Lbound}
\end{align}
After taking the square root (note that $x^\top L x>0$ for all $x\notin\calC$) in (\ref{eqn_prf_tm:main_vector_suff_Lbound}) and substituting the result in (\ref{eqn_prf_tm:main_vector_suff_abound_step2}), the result
\begin{align}
a \leq -\frac{c_1\sqrt{c}}{2}\frac{\sqrt{x^\top\hat{L}x}}{\sqrt{x^\top\hat{L}x}} = -\frac{c_1\sqrt{c}}{2}
\end{align}
follows, which proves finite-time convergence to consensus by Proposition~4 in \cite{Cortes2006} and Lemma~\ref{lm:asymptotic1}.

\textit{(ii)} By using a similar reasoning, we have that for any $a\in \calL_{\calF[\hat{h}]}V(x)$, it satisfies that  $a\leq -c_1$ where $c_1$ satisfying $L_{x_i}-c_1I\geq 0$ for all $i\in\calI$. This again implies finite time convergence.

\end{proof}

\begin{remark}\label{re:conjecture}
Theorem~\ref{thm:main1} provides sufficient conditions for finite-time convergence of the protocol \eqref{e:controller1} satisfying Assumption \ref{ass_vectorfield_1}. However, we conjecture these sufficient conditions to be necessary as well. Namely, for the case $|\calI|>2$, we expect that all the Filippov solutions of \eqref{e:inclusion1} converge to consensus in finite time if and only if $|\calI_c|=1$; and for the case $|\calI|=2$, we expect that the finite-time synchronization is achieved if and only if $|\calI_c|\leq 1$. We show that the latter statement holds according to the following argument.

If $|\calI|=2$ and $|\calI_c|>1$, then $\calI=\calI_c$. In this case, we can only have asymptotic convergence if $\calI=\calI_c$. Indeed, the right-hand side of \eqref{e:closedloop1} is Lipschitz; therefore, finite-time convergence to an equilibrium can not occur. 

Unfortunately, for the case $|\calI|>2$, we can not prove the necessity, which leaves it as an open question.
\end{remark}

We close this subsection by demonstrating the result in Theorem~\ref{thm:main1} and conjecture in Remark \ref{re:conjecture} by an example.

\begin{example}\label{ex:global}

Consider the three-agent system
\begin{align}\label{e:ex_FTC1}
\dot{x}_1 & = L_{x_1}(x_2-x_1) \nonumber\\
\dot{x}_2 & = L_{x_2}\sign(x_1+x_3-2x_2)\\
\dot{x}_3 & = L_{x_3}\sign(x_2-x_3),\nonumber
\end{align}
defined on a line graph with. Notice that this system meets condition
\textit{(i)} in Theorem \ref{thm:main1}.  A phase portrait and trajectory of this system are depicted in Fig.~\ref{fig:ex_global}. There, we can see that finite-time consensus is achieved.

\begin{figure}[t!]
\centering
\begin{subfigure}[t]{0.4\textwidth}
\includegraphics[width=1\textwidth]{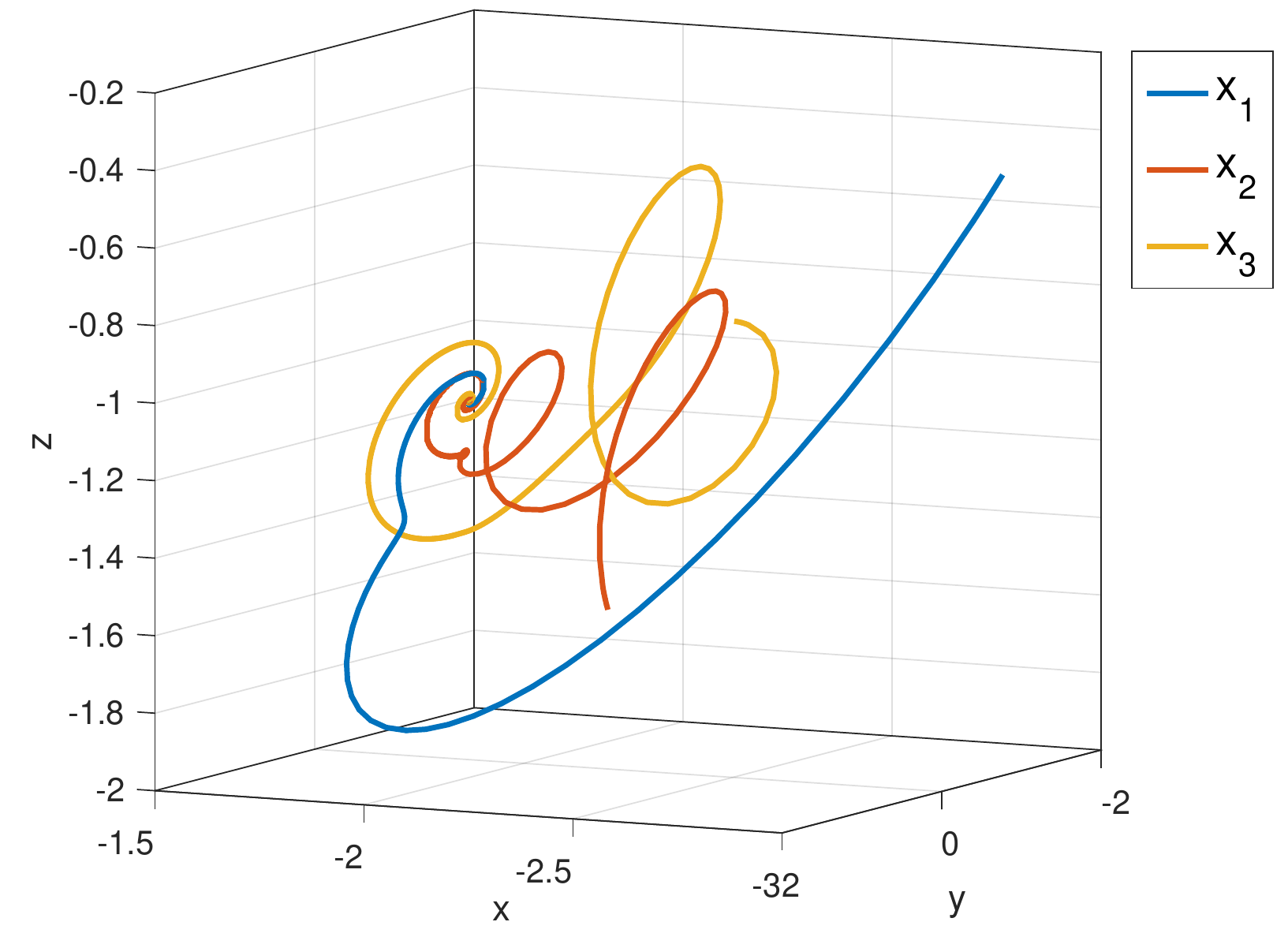}
\caption{Phase portraits of $x_i,i=1,2,3$ in $\R^3$.}\label{fig:ex_FTC_component_sign_div1}
\end{subfigure}
~
\begin{subfigure}[t]{0.4\textwidth}
\includegraphics[width=1\textwidth]{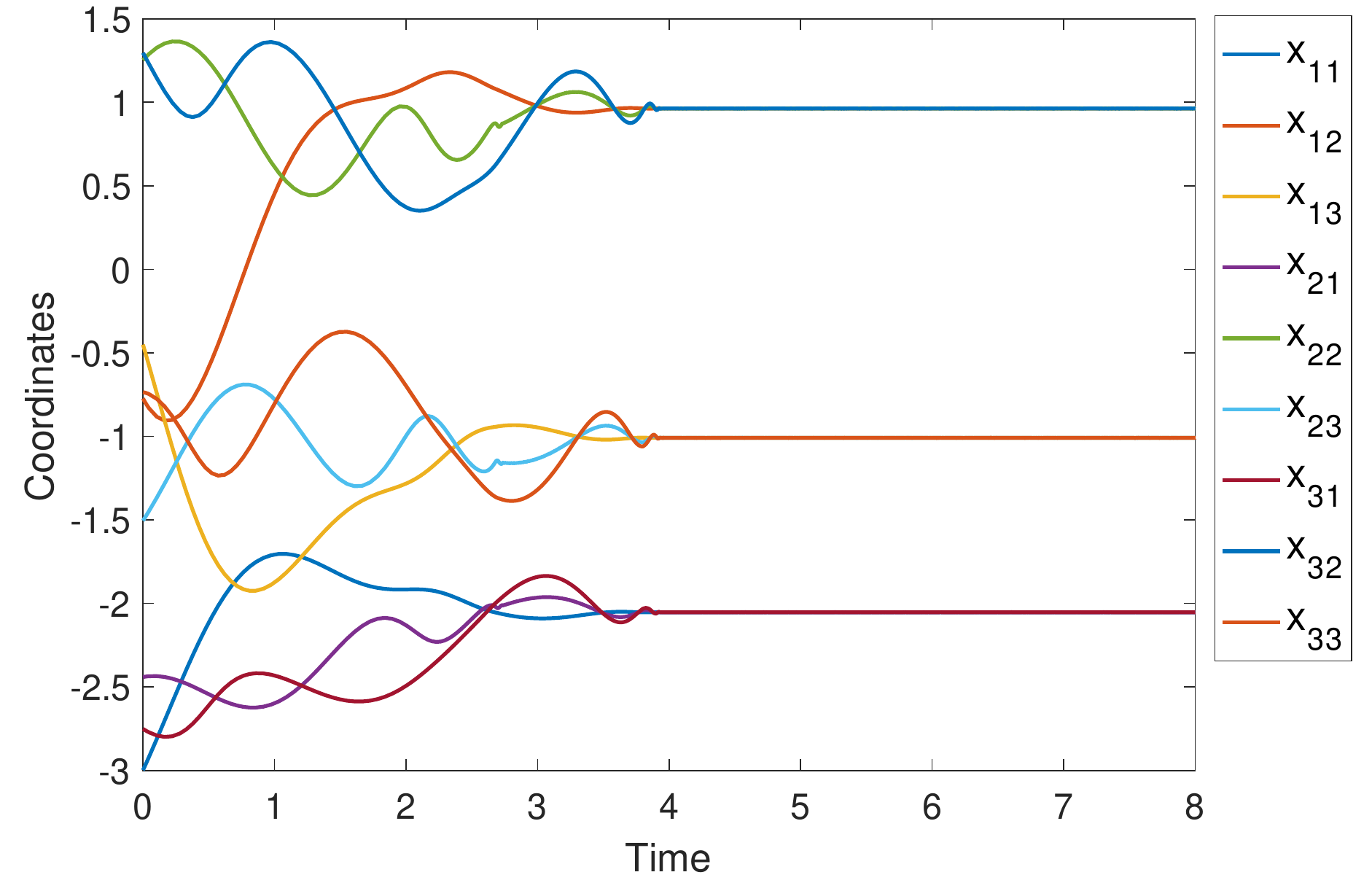}
\caption{Evolution of the three coordinates of $x_i,i=1,2,3$. Finite-time consensus is achieved.}\label{fig:ex_FTC_component_sign_div2}
\end{subfigure}
\caption{The simulation of Example \ref{ex:global}.}\label{fig:ex_global}
\end{figure}

Next, modify the system to
\begin{align}\label{e:ex_AC1}
\dot{x}_1 & = L_{x_1}(x_2-x_1) \nonumber\\
\dot{x}_2 & = L_{x_2}\sign(x_1+x_3-2x_2)\\
\dot{x}_3 & = L_{x_3}(x_2-x_3). \nonumber
\end{align}
Notice that $\calI_c=\{1,3\}$, hence the conditions \textit{(i)} and \textit{(ii)} in Theorem \ref{thm:main1} are not satisfied. As stated in Remark \ref{re:conjecture}, we expect there are some trajectories that only converge to consensus asymptotically, but not in finite time. In fact, we construct such a solution as follows.
For the initial condition satisfying $x_1(0)+x_3(0)=x_2(0)=\mathbf{0}$, the trajectory $x_1(t)=x_1(0)e^{-t}, x_2(t)=\mathbf{0}, x_3(t)=x_3(0)e^{-t}$ is a Filippov solution of \eqref{e:ex_AC1}. Indeed, the trajectory obeys the dynamic
\begin{equation*}
\begin{pmatrix}
\dot{x}_1 \\ \dot{x}_2 \\ \dot{x}_3
\end{pmatrix}
 = \begin{pmatrix}
-x_1 \\ \mathbf{0} \\ -x_3
\end{pmatrix}
\in
\calF \begin{bmatrix}
\begin{pmatrix}
L_{x_1}(x_2-x_1) \\ L_{x_2}\sign(x_1+x_3-2x_2) \\ L_{x_3}(x_2-x_3)
\end{pmatrix}
\end{bmatrix}(x)
\end{equation*}
where we used  that $L_{x_2}=I, L_{x_i}x_i=x_i$, and $\mathbf{0}\in\calF[\sign](\mathbf{0})$. Hence we have a trajectory converging to consensus only asymptotically.
\end{example}

\subsection{Control law for local convergence}\label{ss:main-local}
%

In this subsection, we consider another type of controller to achieve finite-time synchronization. The controller has a finite number of control actions.
However, differently than controller \eqref{e:controller1}, 
the controller in this subsection only guarantees local convergence.

We consider the discontinuous control protocol
\begin{equation}\label{e:controller2}
\omega_i = \sum_{j\in N_i} \sign_c(x_j-x_i)
\end{equation}
where $\sign_c$ is defined  in \eqref{e:signc}. Notice that each $\omega_i$ only takes a finite number of values.
Now the closed-loop dynamic is
\begin{equation}\label{e:system_ab}
\dot{x}_i = L_{x_i}\sum_{j\in N_i} \sign_c(x_j-x_i).
\end{equation}
The compact version of the system \eqref{e:system_ab} can be written as
\begin{equation}\label{e:system_ab_compact}
\dot{x} = - L_x \hat{B} \sign_c\big(\hat{B}^\T x\big),
\end{equation}
where $B$ is the incidence matrix of the underlying graph and $\hat{B}=B\otimes I_3$.
Similar to the previous subsection, we understand the solution of \eqref{e:system_ab_compact} in the sense of Filippov, namely solutions of the following differential inclusion:
\begin{equation}\label{e:system_ab_compact_fili}
\begin{aligned}
\dot{x} &\in \calF[ - L_x \hat{B} \sign_c\big(\hat{B}^\T x \big) ](x) \\
& = - L_x \hat{B} \calF[\sign_c(\hat{B}^\T x)] (x)\\
& =: \calF_3(x),
\end{aligned}
\end{equation}
where the first equality is based on  Theorem 1(5) in \cite{paden1987} and the fact that $L_{x_i}$ is continuous for $\|x_i\|\in [0,\pi)$.
By using Theorem 1 \cite{paden1987}, we can enlarge the differential inclusion $\calF_3$ as follows:
\begin{align}\label{e:system_ab_compact_fili_bigger}
\calF_3(x) & \subset - L_x \hat{B} \bigtimes_{i=1}^{3n} \calF[\sign_c]((\hat{B}^\top)_i x )\nonumber \\
& =: \calF_4(x),
\end{align}
where $(\hat{B}^\top)_i$ is the $i$th row of $\hat{B}^\T$ and the set-valued function $\calF[\sign_c]$ is defined as
\begin{equation}
\calF[\sign_c](x) =
\begin{cases}
1 & \textrm{ if } x>0,\\
[-1,1] & \textrm{ if } x=0, \\
-1 & \textrm{ if } x<0.
\end{cases}
\end{equation}

Before we show the main result, we present a compact strongly invariant set.

\begin{lemma}\label{l: invariant_set_control2}
The set $\calS_2(C) = \{ x\in\R^{3n}| \sum_{i=1}^{n}\|x_i\|_2^2< C \}$ with $C< 4\pi^2$ is strongly invariant for the differential inclusion \eqref{e:system_ab_compact_fili}. Moreover, all the solutions of system \eqref{e:system_ab_compact_fili} converge to consensus asymptotically.
\end{lemma}

\begin{proof}
Consider the Lyapunov function candidate $V(x)=\frac{1}{2}x^\top x=\frac{1}{2}\sum_{i=1}^{n}x^\top_i x_i$. We shall show that the conclusion holds for the bigger inclusion $\calF_4$ defined in \eqref{e:system_ab_compact_fili_bigger}.
	
Since $V$ is smooth, the set-valued Lie-derivative $\calL_{\calF_4}V(x)$ is given as
	\begin{equation}
	\begin{aligned}
	\calL_{\calF_4} V(x) & = x^\top \calF_4(x) \\
	& = -x^\top \hat{B} \bigtimes_{i=1}^{3n} \calF[\sign_c]((\hat{B}^\top)_i x ),
	\end{aligned}
	\end{equation}
	where the last equality is implied by that $L_{x_i}$ is well-defined when $\|x_i\|<2\pi$, which is satisfied by the elements in $\calS_2(C)$, and $x^\top_iL_{x_i}=x^\top_i$. Furthermore, note that
	\begin{equation}\label{e:controller2_invariant_proof}
	\begin{aligned}
	& -x^\top \hat{B} \bigtimes_{i=1}^{3n} \calF[\sign_c]((\hat{B}^\top)_i x ) \\ & =  - \sum_{(i,j)\in\calE} (x_i-x_j)^T \bigtimes_{k=1}^{3}\calF[\sign_c](x_{i_k}-x_{j_k}) \\
	& \subset  \R_{\leq 0},
	\end{aligned}
	\end{equation}
	which indicates that $V(x(t))$ is not increasing along the trajectories when $C<4\pi^2$. Hence the set $\calS_2(C)$ is strongly invariant. Notice that the boundedness of the trajectories is also guaranteed.
	
	Finally, by Theorem 3 in \cite{Cortes2006}, we have that the Filippov solution of system \eqref{e:system_ab_compact_fili} will asymptotically converge to the set
	\begin{equation}
	\Omega  = \overline{\big\{ x\in\R^{3n} \;\big|\; 0\in\calL_{\calF_4}V(x) \big\}}.
	\end{equation}
	By \eqref{e:controller2_invariant_proof} it is straightforward to verify that $\Omega=\calC$. Then the conclusion follows.
\end{proof}

From the previous lemma, we note that the continuity assumption, i.e., Assumption \ref{ass_vectorfield_1} \textit{(i)}, is not needed for controller \eqref{e:controller2}. However, the control law \eqref{e:controller2} can only guarantee local convergence as indicated in the following Theorem and the complete proof can be found in \cite{Wei2017}.

\begin{theorem}
	Assume that the initial rotations of the agents satisfy $\sum_{i=1}^nd_R^2(I,R_i(0))<\pi^2$ and the underlying graph $\calG$ is connected. Consider the multi-agent system \eqref{e:system_ab_compact} and the corresponding differential inclusion \eqref{e:system_ab_compact_fili}. Then, attitude synchronization is achieved in finite time.
\end{theorem}

\section{Conclusion}\label{s:conclusion}

In this paper, we considered the finite-time attitude synchronization problem of multi-agent systems. Two finite-time consensus control protocols were proposed. The first protocol guaranteed global convergence in the sense that the initial rotation of each agent can be arbitrary in an open ball of radius $\pi$, which contains all but a set of measure zero of the rotations in $SO(3)$. In addition, we proposed a second protocol based on binary control, which achieves local convergence to the consensus subspace, in the sense that the initial rotations have to be close enough to the origin in $SO(3)$. For these two controllers, sufficient conditions were presented to guarantee finite-time convergence and boundedness of the solutions. Future studies include further investigation on the necessity of these conditions. Furthermore, the results in this paper based on absolute rotation measurements of the agents, hence finite-time synchronization protocols using relative measurements is another future topic.



\bibliographystyle{plain} 
\bibliography{ref,C:/Users/bartb/Documents/Literature/all}

\end{document}